\documentclass[a4paper,12pt]{article}

\usepackage[top=2.5cm,bottom=2.5cm,left=2.5cm,right=2.5cm]{geometry}

\usepackage{amsmath, amssymb, amsthm}
\usepackage{tikz}
\usetikzlibrary{arrows.meta} 

\usepackage{mathtools}
\usepackage{latexsym}

\usepackage{cite}
\usepackage{url}

\usepackage{enumerate}
\usepackage{appendix}

\numberwithin{equation}{section}

\newtheoremstyle{nopunct}%
  {}%
  {}%
  {\itshape}%
  {}%
  {\bfseries}%
  {.}%
  {.5em}%
  {}%

\theoremstyle{nopunct}
\newtheorem{theorem}{Theorem}[section]
\newtheorem{proposition}[theorem]{Proposition}

\newtheorem{conjecture}[theorem]{Conjecture}
\newtheorem{claim}{Claim}
\theoremstyle{remark}

\usepackage{indentfirst}
\usepackage[section]{placeins}
\allowdisplaybreaks

\usepackage{cleveref}
\crefname{theorem}{Theorem}{Theorems}
\crefname{proposition}{Proposition}{Propositions}
\crefname{lemma}{Lemma}{Lemmas}
\crefname{corollary}{Corollary}{Corollaries}
\crefname{conjecture}{Conjecture}{Conjectures}
\crefname{claim}{Claim}{Claims}
\crefname{remark}{Remark}{Remarks}
\crefname{problem}{Problem}{Problems}
\crefname{example}{Example}{Examples}
\crefname{definition}{Definition}{Definitions}

\begin{document}
\vspace*{10mm}

\begin{center}
\section*{Saturation numbers of $K_3 \vee P_k$}
\vspace{2mm}
 Huaqin Li, Lihua You*, Xiaoxue Zhang\\
\vspace{2mm}
School of Mathematical Sciences\\
South China Normal University, Guangzhou, 510631, P. R. China\\
{\tt 2024022055@m.scnu.edu.cn}, {\tt ylhua@scnu.edu.cn}, {\tt zhang\_xx1209@163.com}
\end{center}
\vspace{7mm}

\noindent \textbf{Abstract:}\quad 
A graph $G$ is said to be $H$-saturated if $G$ contains no copy of $H$, but $G+e$ contains a copy of $H$ for any edge $e$ of $\overline{G}$. The saturation number of $H$, denoted by $\mathrm{sat}(n, H)$, is the minimum number of edges among all $H$-saturated graphs of order $n$. In this paper, we investigate the saturation number of $K_3 \vee P_k$ for $k \ge 3$ and show that $\mathrm{sat}(n, K_3 \vee P_k) = 3n - 6 + \mathrm{sat}(n-3, P_k)$ holds for all $n \ge a_k + 3$ and $k \ge 3$, which verifies the conjecture [Discrete Appl. Math, 394(2026) 62-80] on $\mathrm{sat}(n, K_s \vee P_k)$ is true for the case $s=3$, where $a_k = k$ for $3 \le k \le 5$, $a_k = 3 \cdot 2^{t-1} - 2$ for $k = 2t \ge 6$, and $a_k = 2^{t+1} - 2$ for $k = 2t + 1 \ge 7$. \\

\noindent \textbf{Keywords:} Saturation number; Saturation graph; Minimal

\section{Introduction}

Throughout this paper, we only consider simple, finite, and undirected graphs.

Let $G=(V(G), E(G))$ be a graph with vertex set $V(G)$ and edge set $E(G)$. The order of $G$ is $|V(G)|$ and the size of $G$ is $|E(G)|$. A path, complete graph, or cycle of order $n$ is denoted by $P_n$, $K_n$ or $C_n$, respectively. For a vertex $u \in V(G)$, the set of neighbors of $u$ in $G$ is denoted by $N_G(u)$, $d_G(u) = |N_G(u)|$ is the degree of $u$ in $G$, and $N_G[u] = N_G(u) \cup \{u\}$ is the closed neighborhood of $u$ in $G$. For $u, v \in V(G)$, the distance between $u$ and $v$ is the length of the shortest path connecting $u$ and $v$ in $G$, and is denoted by $d_G(u, v)$. The diameter of $G$ is the greatest distance between two vertices of $G$, denoted by $\operatorname{diam}(G)$. Let $\delta(G)$ and $\Delta(G)$ be the minimum degree and maximum degree of vertices in $G$, respectively. When no ambiguity arises, we use $N(u)$, $d(u)$, $N[u]$, $d(u, v)$, $\delta$ and $\Delta$ instead of $N_G(u)$, $d_G(u)$, $N_G[u]$, $d_G(u, v)$, $\delta(G)$ and $\Delta(G)$, respectively. A vertex of degree zero is called an isolated vertex. An independent set in a graph is a set of vertices no two of which are adjacent. For $X \subseteq V(G)$, we use $G[X]$ to denote the subgraph of $G$ induced by $X$. For $X \subseteq V(G)$ and $Y \subseteq V(G) \setminus X$, $E_G[X, Y]$ is the set of edges of $G$ between $X$ and $Y$, denoted simply by $E[X, Y]$. The union of two graphs $G_1$ and $G_2$ is the graph $G_1 \cup G_2$ with vertex set $V(G_1) \cup V(G_2)$ and edge set $E(G_1) \cup E(G_2)$. The join of graphs $G_1$ and $G_2$ is obtained from $G_1 \cup G_2$ by joining each vertex of $G_1$ to each vertex of $G_2$, denoted by $G_1 \vee G_2$.

A graph $H$ is called a subgraph of $G$ if $V(H) \subseteq V(G)$ and $E(H) \subseteq E(G)$. A copy of $H$ in $G$ is a subgraph of $G$ which is isomorphic to $H$. A graph $G$ is called $H$-saturated if $G$ contains no copy of $H$, but $G+e$ contains a copy of $H$ for any edge $e \in E(\overline{G})$. The saturation number of $H$, denoted by $\operatorname{sat}(n, H)$, is the minimum number of edges among all $H$-saturated graphs of order $n$. If $G$ is an $H$-saturated graph with $\operatorname{sat}(n, H)$ edges, we say $G$ is a minimal $H$-saturated graph.

The concept of the saturation number was introduced by Erd\H{o}s, Hajnal, and Moon ~\cite{erdos1964} in 1964. They established the exact value for complete graphs as $\operatorname{sat}(n, K_p) = (p-2)(n-p+2) + \binom{p-2}{2}$ and characterized $K_{p-2} \vee \overline{K}_{n-p+2}$ as the unique minimal saturated graph. Since then, many exact saturation numbers are now known. These include stars, paths, and matchings~\cite{kaszonyi1986}; cycles~\cite{chen2009,chen2011,furedi2013,gould2006,lan2025,ollmann1972,tuza1989}; complete bipartite and multipartite graphs~\cite{bohman2010,chen2014,gould2007,huang2024}; nearly complete graphs~\cite{faudree2013}; linear forests~\cite{cao2023,chen2015,fan2015}; vertex-disjoint cliques~\cite{chen2024,faudree2009,zhu2025}; regular saturated graphs~\cite{gerbner2022,timmons2022}; hypergraphs~\cite{pikhurko2004}; books and generalized books~\cite{chen2008} and so on. More results can be found in~\cite{currie2021}.

The join of graphs is an important binary operation in graph theory, and the \linebreak saturation numbers concerning the joins of certain specific graphs have attracted \linebreak considerable attention. In 2024, Hu, Luo, and Peng~\cite{hu2024} showed that if $D$ is a graph without isolated vertices, then $\operatorname{sat}(n, K_s \vee D) = \binom{s}{2} + s(n-s) + \operatorname{sat}(n-s, D)$
holds when $s \geq 1$ and $n \geq 3s^2 - s + 2\operatorname{sat}(n-s, D) + 1$. However, Qiu, He, Lu, and Xu~\cite{qiu2026} pointed out that under the condition $n \geq 3s^2 - s + 2\operatorname{sat}(n-s, D) + 1$, the graph $D$ must contain a $K_2$. Therefore, it is interested to investigate the saturation \linebreak numbers when $D$ does not contain a $K_2$. Qiu, He, Lu, and Xu~\cite{qiu2026} proved that $\operatorname{sat}(n, K_1 \vee C_k) = n - 1 + \operatorname{sat}(n-1, C_k)$ for $k \geq 8$ and $n \geq 56k^3$. When $k=4$, Song, Hu, Ji, and Cui~\cite{song2026} proved that $\operatorname{sat}(n, K_1 \vee C_4) = n - 1 + \operatorname{sat}(n-1, C_4)$ for $n \geq 6$. In 2025, Hu, Ji, and Cui~\cite{hu2025} proved that $\operatorname{sat}(n, K_1 \vee P_k) = n - 1 + \operatorname{sat}(n-1, P_k)$ for $k \geq 5$ and sufficiently large $n$. Furthermore, in 2026, Zhang, You, and Zhao~\cite{zhang2025} proved that $\operatorname{sat}(n, K_2 \vee P_k) = 2n - 3 + \operatorname{sat}(n - 2, P_k)$ for $n \geq a_k + 2$ and $k \geq 3$, and proposed a conjecture about $\operatorname{sat}(n, K_s \vee P_k)$ as follows.

 \begin{conjecture} [\hspace{-0.5pt}\cite{zhang2025}]
Let $k,s \geq 3$, $n \geq a_k + s$, where
\begin{equation} \label{eq1.1}
a_k = \begin{cases} 
k, & \text{if } 3 \leq k \leq 5; \\
 3\cdot 2^{t-1} - 2, & \text{if } k \geq 6 \text{ and } k = 2t; \\
2^{t+1} - 2, & \text{if } k \geq 7 \text{ and } k = 2t+1.
\end{cases}
\end{equation}
Then $\operatorname{sat}(n, K_s \vee P_k) = \binom{s}{2} + s(n-s) + \operatorname{sat}(n-s, P_k)$.
\end{conjecture}

In this paper, we improve upon the methods used to determine $\operatorname{sat}(n, K_2 \vee P_k)$~\cite{zhang2025} and verify Conjecture 1.1 is true for the case $s=3$. The following is our main result.

\begin{theorem} \label{thm1.2}
Let $k \geq 3$, $n \geq a_k + 3$, where $a_k$ is defined as \textup{(\ref{eq1.1})}.
Then
\begin{equation} \label{eq1.2}
\operatorname{sat}(n, K_3 \vee P_k) = 3n - 6 + \operatorname{sat}(n-3, P_k) = \begin{cases} 
\left\lfloor \frac{7n-15}{2} \right\rfloor, & \text{if } k=3; \\
\frac{7n-15}{2}, & \text{if } k=4, n \text{ is odd}; \\
\frac{7n-12}{2}, & \text{if } k=4, n \text{ is even}; \\
\left\lceil \frac{23n-55}{6} \right\rceil, & \text{if } k=5; \\
4n - 9 - \left\lfloor \frac{n-3}{a_k} \right\rfloor, & \text{if } k \geq 6.
\end{cases}
\end{equation}
\end{theorem}

The rest of this paper is organized as follows. In Section 2, we present some known results that are useful for the proof. In Section 3, we prove Theorem~\ref{thm1.2}. In Section 4, we conclude this paper.

\section{Preliminaries}

In this section, we present some results and notation that will be used later. For convenience, we refer to $K_3$ and $P_k$ as the center and path part of $K_3 \vee P_k$, respectively, and use $C(K_3 \vee P_k)$ and $PP(K_3 \vee P_k)$ to denote the sets of vertices in the center and the path part of $K_3 \vee P_k$, respectively.

\begin{theorem}[\hspace{-0.5pt}\cite{kaszonyi1986}]\label{thm2.1}
Let $k \geq 3$ and $a_k$ be defined as \textup{(\ref{eq1.1})}. Then we have \\
\textup{(i)} $\operatorname{sat}(n, P_3) = \left\lfloor \frac{n}{2} \right\rfloor$ for $n \geq 3$. \\
\textup{(ii)} $\operatorname{sat}(n, P_4) = \begin{cases} \frac{n}{2}, & \text{for } n \geq 4, n \text{ is even}; \\ \frac{n+3}{2}, & \text{for } n \geq 5, n \text{ is odd}. \end{cases}$ \\
\textup{(iii)} $\operatorname{sat}(n, P_5) = \left\lceil \frac{5n-4}{6} \right\rceil$ for $n \geq 5$. \\
\textup{(iv)} $\operatorname{sat}(n, P_k) = n - \left\lfloor \frac{n}{a_k} \right\rfloor$ for $n \geq a_k$ and $k \geq 6$.
\end{theorem}

\begin{theorem}[\hspace{-0.5pt}\cite{zhang2025}]\label{thm2.2}
Let $k \geq 3$, $n \geq a_k + 2$, where $a_k$ is defined as \textup{(\ref{eq1.1})}. Then we have $\text{sat}(n, K_2 \vee P_k) = 2n - 3 + \text{sat}(n - 2, P_k)$.
\end{theorem}

\begin{proposition}[\hspace{-0.5pt}\cite{zhang2025}]\label{proposition:2.3}
\textit{For $n \geq k+s$ and $k, s \geq 3$, $\text{sat}(n, K_s \vee P_k) \leq \binom{s}{2} + s(n-s) + \text{sat}(n-s, P_k)$.}
\end{proposition}

\begin{proposition}[\hspace{-0.5pt}\cite{zhang2025}]\label{proposition:2.4}
\textit{Let $G$ be a $K_1 \vee H$-saturated graph of order $n$. If $\Delta(G) = n-1$, then $|E(G)| \geq n-1 + \text{sat}(n-1, H)$, the equality holds if and only if $G \cong K_1 \vee F$, where $F$ is a minimal $H$-saturated graph of order $n-1$.}
\end{proposition}

\begin{proposition}\label{proposition:2.5}
\textit{Let $G$ be a $K_s \vee P_k$-saturated graph with $s \geq 2, k \geq 3$, $x, y \in V(G)$, $xy \notin E(G)$, and $Q$ be a copy of $K_s  \vee P_k$} in $G+xy$. Then we have\\
\textup{(1)} there exists $V_1 \subseteq V(G)$, such that $\{x, y\} \cup C(Q) \subseteq V_1$, and $G[V_1] + xy\cong K_{s+2}$.\\
\textup{(2)} there exist at least s paths of length $2$ from $x$ to $y$ in $G$.\\
\textup{(3)} $diam(G)=2$ and $\delta (G) \geq s$.\\
\textup{(4)}  $C(Q) \subseteq \{x, y\} \cup (N(x) \cap N(y))$.    
\end{proposition}

\noindent \textbf{\textit{Proof.}} Let \begin{equation*} V_1 = \begin{cases} 
C(Q) \cup \{z_1, z_2\}, \text{ where } z_1, z_2 \in PP(Q), z_1 z_2 \in E(Q), & \text{if } x, y \in C(Q); \\
C(Q) \cup \{y, z\}, \text{ where } z \in PP(Q) \cap N(y), & \text{if } x \in C(Q), y \in PP(Q); \\
C(Q) \cup \{x, y\}, & \text{if } x, y \in PP(Q). \\
\end{cases}
\end{equation*}

Clearly, $\{x, y\} \cup C(Q) \subseteq V_1 \subseteq V(G)$ and $G[V_1] + xy\cong K_{s+2}$, and thus (1) holds. By (1) and direct checking, (2), (3), and (4) hold immediately.
\hfill $\square$

\section{Proof of Theorem~\ref{thm1.2}}

For the cases $k=1,2$, we note that $K_3 \vee P_1 \cong K_4$ and $K_3 \vee P_2 \cong K_5$, the results for which have already been established in~\cite{erdos1964}. In this section, we prove Theorem~\ref{thm1.2}. 

\noindent \textbf{\textit{Proof of Theorem~\ref{thm1.2}.}} Let $G$ be a $K_3 \vee P_k$-saturated graph with $k \geq 3$ of order $n$. In order to show \textup{(\ref{eq1.2})}, by Theorem~\ref{thm2.1} and Proposition \ref{proposition:2.3}, we only need to show 
\begin{equation}\label{eq3.1}
|E(G)| \geq 3n - 6 + \operatorname{sat}(n-3, P_k).
\end{equation}

Let $v \in V(G)$ with $d(v)=\delta(G)=\delta$ such that $|E(G[N(v)])|$ is minimized, $N(v)=\{v_1,v_2,\dots,v_{\delta}\}$ and $U=V(G)\setminus N[v]$. Then $\delta \geq 3$ and $|E[\{u\}, N(v)]|\neq \emptyset$ for any $u\in U$  by (3) of Proposition \ref{proposition:2.5} and $s = 3$. 

For any $v_i \in N(v)$, we call $v_i$ a $K_3$-vertex if there exist $v', v'' \in N(v)$ such that $G[\{v_i, v', v''\}] \cong K_3$. Otherwise, we call $v_i$ a non-$K_3$-vertex. 

Let $V_{K_3}$ and $V_{NK_3}$ be the sets of all $K_3$-vertices and all non-$K_3$-vertices in $N(v)$, respectively, and let $G[V_{K_3}] = G_{K_3}$, and $G[V_{NK_3}] = G_{NK_3}$. Then $V_{K_3}\cap V_{NK_3}= \emptyset$, and $V_{K_3}\cup V_{NK_3}= N(v)$. 

For any $u \in U$, $\{t_1, t_2, \dots, t_p\} \subseteq \{1, 2, \dots, \delta\}$, we define 

\[
U_p = \{u \in U \mid |N(u) \cap N(v)| = p\}, \quad U_{\geq p}=\bigcup_{q=p}^{\delta}U_q,
\]
 \[
X_{t_1 t_2 \cdots t_p} = \{u \mid u \in U_p \cap N(v_{t_1}) \cap N(v_{t_2}) \cap \cdots \cap N(v_{t_p})\}.
\]
 Then $3 \leq p \leq \delta$ by $d(v)= \delta$ and (2) of Proposition \ref{proposition:2.5} with $s=3$. Furthermore, $U_p = \bigcup\limits_{\{t_1, t_2, \dots, t_p\} \subseteq \{1, 2, \dots, \delta\}} X_{t_1 t_2 \cdots t_p}$, $U = \bigcup_{p=3}^{\delta} U_p$, $|E[N(v), U_p]|=|E[N[v], U_p]| = p|U_p|$, $|E[N(v), U]|=|E[N[v], U]| = \sum_{p=3}^{\delta} p|U_p|$, 
and
\begin{flalign}
    |E(G)|=|E(G[N[v]])| + \sum_{p=3}^{\delta} p|U_p| + |E(G[U])|. \text{ } \label{eq3.2}  
\end{flalign}

\begin{claim}\label{claim:1}

Let $3 \leq p \leq \delta$. The following hold.\\
\textup{(1)} If $G[\{v_{t_1}, v_{t_2}, \dots, v_{t_p}\}]$ contains no copy of $K_3$, where $\{t_1, t_2, \dots, t_p\} \subseteq \{1, 2, \dots, \delta\}$, then $X_{t_1 \dots t_p} = \emptyset$.\\
\textup{(2)} If $u\in U $, then $G[N(u) \cap N(v)]$ contains a copy of $K_3$.\\
\textup{(3)} $G[N(v)]$ contains a copy of $K_3$.\\
\textup{(4)} $|E[V_{K_3}, U_p]| \geq 3|U_p|$ and $E[V_{NK_3}, U_3] = \emptyset$.
\end{claim}

\noindent \textbf{\textit{Proof of \cref{claim:1}.}} Suppose $u \in X_{t_1 \dots t_p} \neq \emptyset$. Since $vu \notin E(G)$, there exists a copy of $K_5$ containing $v$ and $u$ in $G + vu$ by (1) of Proposition \ref{proposition:2.5} and $s = 3$. 
This implies that $G[\{v_{t_1}, v_{t_2}, \dots, v_{t_p}\}]$ contains a copy of $K_3$, which is a contradiction. Thus (1) holds.

Now we show that (2) holds. If $u \in U$, then there exist $t_1, t_2, \dots, t_p \in \{1, 2, \dots, \delta\}$ such that $3 \leq p \leq \delta$ and $u \in X_{t_1 \dots t_p} \subseteq U_p \subseteq U$. Hence $N(u) \cap N(v) = \{v_{t_1}, v_{t_2}, \dots, v_{t_p}\} $ and $G[N(u) \cap N(v)]$ contains a copy of $K_3$ by (1). 

Now we show (3) holds. If $G[N(v)]$ is complete, then (3) holds; if $U \neq \emptyset$, then (3) holds by (2); if $U=\emptyset$ and $G[N(v)]$ is not complete, then there exist $v_j, v_s \in N(v)$ such that $v_j v_s \notin E(G)$,  and thus there exists a copy of $K_5$ containing $v_j$ and $v_s$ in $G + v_j v_s$  by (1) of Proposition \ref{proposition:2.5} and $s = 3$, which implies that $G[N(v)]$ contains a copy of $K_3$. 

For any $u \in U_p$, we have $|E[V_{K_3}, \{u\}]| \geq 3$ by (2), which implies $|E[V_{K_3}, U_p]| \geq 3|U_p|$ and $|E[V_{NK_3}, U_3]| = 3|U_3|-|E[V_{K_3}, U_3]| \leq (3-3)|U_3|=0$. Thus (4) holds. \hfill $\square$

\begin{claim}\label{claim:2}
Let $v_i \in N(v)$. If $G[N(v) \setminus \{v_i\}]$ contains no copy of $K_3$ and $d_{G[N(v)]}(v_i) = \delta - 1$, then $d(v_i) = n - 1$.
\end{claim}

\noindent \textbf{\textit{Proof of \cref{claim:2}.}} We show $d(v_i)=n-1$ by considering $U = \emptyset$ and $U \neq \emptyset$.

If $U=\emptyset$, then $d(v_i)=1+(\delta-1)=\delta=n-1$ since $d_{G[N(v)]}(v_i) = \delta - 1$. 

If $U \neq \emptyset$, we need to show $u \in N(v_i)$ for any $u \in U$. By (2) of \cref{claim:1}, we have $G[N(u) \cap N(v)]$ contains a copy of $K_3$ for any $u \in U$. Thus the copy must contain $v_i$ since $G[N(v) \setminus \{v_i\}]$ contains no copy of $K_3$, which implies that $u \in N(v_i)$ and thus $d(v_i) = 1+(\delta-1)+(n-\delta-1) = n - 1$ since $d_{G[N(v)]}(v_i) = \delta - 1$.
\hfill $\square$

By (4) of \cref{claim:1}, we have $|E[V_{NK_3}, U_p]| = p|U_p| - |E[V_{K_3}, U_p]| \leq (p - 3)|U_p|$ and $|E[V_{NK_3}, U]|=|E[V_{NK_3}, U_3]|+\sum\limits_{p=4}^{\delta} |E[V_{NK_3}, U_p]| \leq 0+\sum\limits_{p=4}^{\delta} (p-3) |U_p|$. Then 
\begin{align}
            \sum_{v_i \in V_{NK_3}} d(v_i) = |V_{NK_3}| + 2|E(G_{NK_3})| + |E[V_{K_3}, V_{NK_3}]| + |E[V_{NK_3}, U]| \geq \delta  |V_{NK_3}|\notag
        \end{align}
\noindent implies that
  \begin{align} &(\delta-1) |V_{NK_3}| + 2|E(G_{K_3})| + |E[V_{K_3}, V_{NK_3}]| \notag \\ 
  \leq & 2|E(G[N(v)])| + |E[V_{NK_3}, U]| \notag \\
  \leq & 2|E(G[N(v)])| + \sum_{p=4}^{\delta} (p-3) |U_p|. \label{eq3.3}
\end{align}

 Define $\delta_U=\min\{d_G(u)\mid u\in U\}$, and $r=|C(Q)\cap U|$, where $Q$ is a copy of $K_3 \vee P_k$ in $G+v_i v_j$ with $v_i,v_j \in N(v)$ and $v_i v_j \notin E(G)$ Then $\delta_U \geq \delta$ and $0 \leq r \leq 3$. Clearly, for any $u \in C(Q) \cap U$, $d_G(u)$=$d_{G+v_i v_j}(u) \geq k+2$. 
 
 If $k+2-\delta^* \geq 0$ with $\delta^* \in \{\delta_U, \delta\}$, then 
\begin{align}
           \sum\limits_{u \in U} d_G(u) = |E[N[v], U]| + 2|E(G[U])| \geq \delta^*(n - \delta - 1) + r(k+2-\delta^*), \label{eq3.4}
        \end{align}

\noindent furthermore, by \eqref{eq3.3}, \eqref{eq3.4}, $|U|=n-\delta-1$ and $\sum\limits_{p=3}^{\delta} p|U_p| = 3|U| + \sum\limits_{p=4}^{\delta} (p-3)|U_p|$, we have
\begin{flalign}
    &2|E(G)| = \sum_{u \in N[v]} d_G(u) + \sum_{u \in U} d_G(u) \notag \\
    =& \big(2|E(G[N[v]])| + |E[N[v], U]|\big) + \big(|E[N[v], U]| + 2|E(G[U])|\big)  \notag \\
    \geq& 2\delta + 2|E(G[N(v)])| + \sum_{p=4}^{\delta} (p-3)|U_p| + (\delta^*+3)(n-\delta-1) + r(k+2-\delta^*) \label{eq3.5} \\
    \geq& 2\delta + (\delta-1) |V_{NK_3}| + 2|E(G_{K_3})| + |E[V_{K_3}, V_{NK_3}]| + (\delta^*+3)(n-\delta-1) + r(k+2-\delta^*), \label{eq3.6} &
    \end{flalign}
    and for any $\delta $, we have
 \begin{flalign}
    &2|E(G)| = \sum_{u \in N[v]} d_G(u) + \sum_{u \in U} d_G(u) \notag & \\
    \geq& 2\delta + 2|E(G[N(v)])| + \sum_{p=4}^{\delta} (p-3)|U_p| + (\delta_U+3)(n-\delta-1) \label{eq3.7} & \\
    \geq& 2\delta + 2|E(G[N(v)])| + \sum_{p=4}^{\delta} (p-3)|U_p| + (\delta+3)(n-\delta-1) \label{eq3.8} & \\
    \geq& 2\delta + (\delta-1) |V_{NK_3}| + 2|E(G_{K_3})| + |E[V_{K_3}, V_{NK_3}]| + (\delta+3)(n-\delta-1). \label{eq3.9} &
\end{flalign}

\begin{claim}\label{claim:3}
Let $ v_i, v_j \in N(v)$ with $v_iv_j \notin E(G)$. If $u \in U$ and $uv_i, uv_j \in E(G)$, then $u \in U_{\ge 4}$. 
\end{claim}

\noindent \textbf{\textit{Proof of \cref{claim:3}.}} 
Clearly, there exist $t_1, t_2, \dots, t_p \in \{1, 2, \dots, \delta\}$ such that $3 \leq p \leq \delta$ and $u \in X_{t_1 \dots t_p} \subseteq U_p \subseteq U$. Then $p\geq 4$ since $v_iv_j \notin E(G)$ and (1) of \cref{claim:1}, which implies $u \in U_{\ge 4}$. 
\hfill $\square$

\begin{claim}\label{claim:4}
Let $ v_i, v_j \in N(v)$ with $v_iv_j \notin E(G)$, $Q$ be a copy of $K_3 \vee P_k$ in $G+v_i v_j$ and $r=|C(Q)\cap U|$. If $N(v_i)\cap N(v_j)\cap N(v)=\emptyset$, then $1\leq r \leq 3$ and 
\begin{equation} \sum_{p=4}^{\delta}(p-3)|U_p|\geq \sum_{p=4}^{\delta}|U_p| \geq \begin{cases} 
k+1,  & \text{if $r = 1$}; \\
3,  & \text{if $r = 2, 3$}. \label{eq3.10}\\
\end{cases}
\end{equation}
\end{claim}

\noindent \textbf{\textit{Proof of \cref{claim:4}.}}
Clearly, $C(Q)\subseteq \{v_i,v_j\}\cup U$ by (4) of Proposition \ref{proposition:2.5} and $N(v_i)\cap N(v_j)\cap N(v)=\emptyset$, which implies $ 1 \leq r \leq 3 $ and $v \notin V(Q)$. Furthermore, $C(Q) \subseteq \{v_i, v_j\} \cup U_{\geq 4}$, and $PP(Q)\subseteq U_{\geq 4}$ when $r =1$ by \cref{claim:3}.

If $ r = 1 $, then $ v_i, v_j \in C(Q) $, $ |PP(Q) \cap U_{\geq 4}| = k $ and $ |C(Q) \cap U_{\geq 4}| = 1 $. Then  $\sum_{p=4}^{\delta} (p-3)|U_p| \geq \sum_{p=4}^{\delta}|U_p| \geq k+1$.

If $ r = 2 $, we have $|\{v_i, v_j\}\cap C(Q)|=1$, $|\{v_i, v_j\}\cap PP(Q)|=1$ and there exists at least one $w \in PP(Q)$ such that $wv_i, wv_j \in E(G)$, which implies $w \in U_{\geq4}$ by \cref{claim:3}, say, $|PP(Q) \cap U_{\geq 4}| \geq 1 $.
Thus $\sum_{p=4}^{\delta} (p-3)|U_p| \geq \sum_{p=4}^{\delta}|U_p| \geq |V(Q) \cap U_{\geq 4}| \geq 1+2 = 3$ since $|C(Q) \cap U_{\geq 4}|=2$.

If $ r = 3 $, then $v_i, v_j \in PP(Q)$, $ |C(Q) \cap U_{\geq 4}| = 3 $ and $\sum_{p=4}^{\delta} (p-3)|U_p| \geq \sum_{p=4}^{\delta}|U_p| \geq |V(Q) \cap U_{\geq 4}| \geq 3$.\hfill $\square$

\begin{claim}\label{claim:5}
Let $ v_i, v_j \in N(v)$ with $v_iv_j \notin E(G)$, $Q$ be a copy of $K_3 \vee P_k$ in $G+v_i v_j$ and $r=|C(Q)\cap U|$. If $E(G[N(v_i)\cap N(v_j)\cap N(v)])=\emptyset$, then $0\leq r \leq 3$ and
\begin{equation} \sum_{p=4}^{\delta}(p-3)|U_p|\geq \sum_{p=4}^{\delta}|U_p| \geq \begin{cases} 
k,  & \text{if $r = 0$}; \\
2,  & \text{if $r = 1,2$}; \\
3, & \text{if $r = 3$}. \label{eq3.11} \\
\end{cases}
\end{equation}
\end{claim}

\noindent \textbf{\textit{Proof of \cref{claim:5}.}} 
If $N(v_i)\cap N(v_j)\cap N(v)= \emptyset$, then \cref{claim:5} holds by \cref{claim:4}.

In the following, we consider $N(v_i)\cap N(v_j)\cap N(v)\neq\emptyset$. 

Clearly, we have $C(Q) \subseteq \{v_i, v_j\} \cup (N(v_i) \cap N(v_j) \cap N(v)) \cup U$, $v \notin Q$ and $ 0 \leq r \leq 3 $  by (4) of Proposition \ref{proposition:2.5}, $N(v_i)\cap N(v_j)\cap N(v)\neq\emptyset$ and $E(G[N(v_i)\cap N(v_j)\cap N(v)])=\emptyset$.
 Furthermore, $C(Q) \subseteq \{v_i, v_j\} \cup (N(v_i) \cap N(v_j) \cap N(v)) \cup U_{\geq 4}$, and $PP(Q)\subseteq U_{\geq 4}$ when $r =0$ by \cref{claim:3}.

If $ r = 0 $, then $ v_i, v_j \in C(Q) $, $ |PP(Q) \cap U_{\geq 4}| = k $. Then $\sum_{p=4}^{\delta} (p-3)|U_p| \geq \sum_{p=4}^{\delta}|U_p| \geq k$.

If $ r = 1 $, then $|\{v_i, v_j\}\cap C(Q)|=2$ or $1$. If $|\{v_i, v_j\}\cap C(Q)|=2$, then $|C(Q) \cap U_{\geq 4}|=1$ and $|PP(Q) \cap U_{\geq 4}|\geq1$; if $|\{v_i, v_j\}\cap C(Q)|=1$, then $|C(Q) \cap U_{\geq 4}|=2$. Therefore, $\sum_{p=4}^{\delta} (p-3)|U_p| \geq \sum_{p=4}^{\delta}|U_p| \geq 2$.

If $ r = \{2, 3\}$, then $|C(Q) \cap U_{\geq 4}| = r $, and thus $\sum_{p=4}^{\delta} (p-3)|U_p| \geq \sum_{p=4}^{\delta}|U_p|\geq  |V(Q) \cap U_{\geq 4}|  \geq r$.
\hfill $\square$

\vspace{0.3cm}
In the following, we divide the remaining proof into five parts based on the values of $\delta$ and $\Delta$ to show \eqref{eq3.1} holds: (1) $ \delta\geq8$ or $ \delta=3$ or $\Delta = n-1$; (2) $\delta=4 \leq \Delta \leq n-2$; (3) $\delta=5 \leq \Delta \leq n-2$; (4) $\delta=6 \leq \Delta \leq n-2$; (5) $\delta=7 \leq \Delta \leq n-2$.

\subsection{$ \delta\geq8$ or $ \delta=3$ or $\Delta = n-1$}
If $\delta\geq8$, then $|E(G)|= \frac{1}{2}\sum\limits_{u \in V(G)} d(u)\geq 4n$. Clearly, \eqref{eq3.1} holds for $\delta\geq8$ by direct checking.

If $\delta=3$, then $d(v_i) = n - 1$ for any $v_i \in N(v)$ by \cref{claim:2}, and thus $\Delta = n-1$.

If $\Delta = n-1$, then $|E(G)| \geq n-1 + \operatorname{sat}(n-1, K_2 \vee P_k) = n-1 + 2(n-1) - 3+\operatorname{sat}(n-3, P_k) = 3n-6 + \operatorname{sat}(n-3, P_k)$ by Proposition~2.4, Theorem~2.2 and $K_3 \vee P_k = K_1 \vee (K_2 \vee P_k)$.

Therefore, \eqref{eq3.1} holds for $ \delta\geq8$ or $ \delta=3$ or $\Delta = n-1$.

\subsection{$\delta=4 \leq \Delta \leq n-2$}

By (3) of \cref{claim:1}, $|E(G[N(v)])| \geq 3$. If $|E(G[N(v)])| = 4$ or $5$, then $\Delta = n-1$ by \cref{claim:1} and \cref{claim:2}, a contradiction with $\Delta \leq n-2$. Thus $|E(G[N(v)])| = 3$ or $6$. Now we show \eqref{eq3.1} holds by the following two cases. 

\noindent \textbf{Case 1:} $|E(G[N(v)])| = 3$. 

By (3) of \cref{claim:1}, $G[N(v)] \cong K_3 \cup K_1$.
Let $V_{K_3}= \{v_1, v_2, v_3\} $. Then $U=X_{1 2 3} \cup X_{1 2 3 4}$ with $U_3=X_{1 2 3}$ and $U_4=X_{1 2 3 4}$ by (1) of \cref{claim:1}. Since $d(v_4)\geq \delta =4$, we have $|X_{1 2 3 4}| \geq 3$, and thus $n \geq 8$.

Now we show $d_G(u)\geq 5$ for any $u \in U$, which implies $\delta_U \geq 5$. 

If $u \in X_{1 2 3}$, then $uv_4 \notin E(G)$, which implies there are three paths of length 2 from $u$ to $v_4$ by (2) of Proposition \ref{proposition:2.5}. It follows that $|E[\{u\}, X_{1 2 3 4}]| \geq 3$, and thus $d_G(u) \geq 6$. 

If $u\in X_{1 2 3 4}$, we consider the copy of $K_3\vee P_k$ in $G+vu$, say $Q$. Clearly, $\{u, v\}\subseteq V(Q)$. We can obtain $v \notin C(Q)$, and thus $v\in PP(Q)$. Otherwise, if $v\in C(Q)$, we have $k=2$,  a contradiction with $k\geq3$. By (4) of Proposition \ref{proposition:2.5} and $G[N(v)]\cong K_3\cup\{v_4\}$, we have $C(Q) \subseteq \{u, v_1, v_2, v_3\}$. If $u\in C(Q)$, then $d_{G[U]} (u)\geq k-2$, and thus $d_G (u)\geq 4+(k-2)\geq 5$. If $u\in PP(Q)$, then $C(Q)=\{v_1,v_2,v_3\}$ since $v\in V(Q)$, and thus $d_{G[U]} (u)\geq 1$, which implies $d_G(u)\geq 5$. 

Since $v_3v_4 \notin E(G)$, $G+v_3v_4$ contains a copy of $K_3\vee P_k$, say $Q_{3,4}$. Let $r = |C(Q_{3,4})\cap U|$. Then $N(v_3)\cap N(v_4)\cap N(v)=\emptyset$, and thus $1\leq r \leq 3$ by \cref{claim:4}.

If $\delta_U \geq 6$, then by \eqref{eq3.7} and \cref{claim:4}, we have

\begin{equation*} 2|E(G)| \geq \begin{cases} 
9n+k-30,  & \text{if $r = 1$}; \\
9n-28, & \text{if $r =2, 3$}. \\
\end{cases}
\end{equation*}
Thus \eqref{eq3.1} holds for $k\geq 3$ by direct checking.

If $\delta_U= 5$, then $k+2-\delta_U\geq 0$ since $k\geq3$. By \eqref{eq3.5} with $\delta^* = \delta_U$ and \cref{claim:4}, we have 

\begin{equation*} 2|E(G)| \geq \begin{cases} 
8n+2k-28,  & \text{if $r = 1$}; \\
8n+2k-29,  & \text{if $r = 2$}; \\
8n+3k-32, & \text{if $r = 3$}. \\
\end{cases}
\end{equation*}
Thus \eqref{eq3.1} holds for $k\geq 3$ by direct checking.

\noindent \textbf{Case 2:} $|E(G[N(v)])| = 6$.

In this case, $G[N(v)] \cong K_4$ and $U=X_{1 2 3} \cup  X_{1 2 4} \cup  X_{1 3 4} \cup  X_{2 3 4} \cup  X_{1 2 3 4}$. We have $X_{abc} \neq \emptyset$ for $1 \leq a < b < c \leq 4$. Otherwise, $d(v_\ell) = n-1$, where $\ell \in \{1,2,3,4\} \setminus \{a,b,c\}$. Then $|U|\geq 4$ and $n\geq 9$. 

Firstly, we show \cref{claim:6} holds.

\begin{claim}\label{claim:6}
Let $k \geq 3$, $\delta = 4 \leq \Delta \leq n-2$, $|E(G[N(v)])| = 6$, and $x, y \in U_3$. Then there exists a path from $x$ to $y$ in $G[U]$. 
\end{claim}

\noindent \textbf{\textit{Proof of \cref{claim:6}.}} 
Without loss of generality, we assume $xy \notin E(G)$, and $\{a, b, c, \ell\}= \{1, 2, 3, 4\}$. 

If $x \in X_{abc}$ and $y \in X_{ab\ell}$, then there exists at least a path from $x$ to $y$ in $G[U]$ by (2) of Proposition \ref{proposition:2.5} and $|N(x)\cap N(y)\cap N(v)|= 2$.

If $x, y \in X_{abc}$, then there exists $z\in X_{ab\ell}$ such that there exist a path from $x$ to $z$ in $G[U]$ and a path from $z$ to $y$ in $G[U]$, and thus there exists a path from $x$ to $y$ in $G[U]$.
\hfill $\square$

Now we complete the proof of Case 2 by the following two subcases.

\noindent \textbf{Subcase 2.1:} $U_4 = \emptyset$.

In this subcase, $G[U]$ is connected by \cref{claim:6}, which implies $|E(G[U])| \geq |U| - 1$. 

If $k=3$, then $|E(G)| \geq 10+ 3(n-5)+ (|U|-1) \geq 4n - 11 \geq \left\lfloor \frac{7n-15}{2} \right\rfloor$ by  \eqref{eq3.2}, and thus \eqref{eq3.1} holds.

If $k \geq 4$, we consider the copy of $K_3 \vee P_k$ in $G+uv$, say $Q$, where $u \in U= U_3$.

Clearly, $v \notin C(Q)$ since $d_{G+uv}(v) = 5 < k+2$. Then $v \in PP(Q)$.

If $u \in C(Q)$, then $|U \cap V(Q)| \geq k-1$ as $u \in U_3$, and $G[U]$ contains a copy of $K_1 \vee P_{k-2}$. We have $|E(G[U])| \geq |U|-1 + (k-3) = n + k - 9$ since $G[U]$ is connected and $G[U]$ contains a copy of $K_1 \vee P_{k-2}$.
By \eqref{eq3.2}, we have $|E(G)| \geq 10+ 3(n-5)+ (n + k - 9) = 4n+k-14$. Thus \eqref{eq3.1} holds for $k\geq 4$ by direct checking.

If $u \in PP(Q)$, then $C(Q)\subseteq N(v)$ and $|U \cap V(Q)| \geq k-2$. Without loss of generality, we assume $u \in X_{123}$, and
thus there exists a path of length $k-3$ with endpoint $u$ in $G[X_{123}]$. Let $P = u u_1 \dots u_{k-3}$ be such a path in $G[X_{123}]$ and $l$ be the length of the longest path in $G[X_{123}]$. Then $k-3 \leq l \leq k-2$.
Otherwise, we have $l\geq k-1$, and thus there exists a copy of $K_3 \vee P_k$ in $G[\{v_1, v_2, v_3\} \cup X_{123}]$, a contradiction.

If $k\in \{4, 5\}$, then $n= 5+ |U| \geq k+6$ since $|X_{123}| \geq k-2$, and thus $|E(G)| \geq 10 +3(n-5)+ |U|-1 \geq 4n-11$ by \eqref{eq3.2}, and thus \eqref{eq3.1} holds by direct checking. 

If $k \geq 6$, we show $G[X_{123}]$ contains a cycle.

If $l = k-3$, then $P$ is the longest path in $G[X_{123}]$. Applying a similar argument to $u_1$, we obtain a path $P_1$ starting at $u_1$ with length $k-3$. Then $P_1 \cup P$ contains a cycle. Otherwise, $P_1 \cup P$ contains a path with length longer than $k-3$ in $G[X_{123}]$, a contradiction.

If $l = k-2$, we assume $P' = w_1 w_2 \dots w_{k-1}$ is the longest path in $G[X_{123}]$. Applying a similar argument to $w_3$, we obtain a path $P'_1$ starting at $w_3$ with length $k-3$. Then $P'_1 \cup P'$ contains a cycle. Otherwise, $P'_1 \cup P'$ contains a path with length longer than $k-2$ in $G[X_{123}]$, a contradiction.

For $u \in X_{124} \cup X_{134}\cup X_{234}$, we have the same conclusion, and then $G[X_{124}]$, $G[X_{134}]$ and $G[X_{234}]$ also contain a cycle. 

Therefore, $|E(G[U])| \geq |U|-1 + 4 \geq n - 2$.  By \eqref{eq3.2}, we have $|E(G)|\geq 4n - 7 > 4n - 9 - \left\lfloor \frac{n-3}{a_k} \right\rfloor
$. Thus \eqref{eq3.1} holds for $k \geq 6$.

\noindent \textbf{Subcase 2.2:} $U_4 \neq \emptyset$.

In this subcase, we have $k \geq 4$ since $G[N[v] \cup U_4]$ contains a copy of $K_3 \vee P_3$.

Let $G_1, \dots, G_m$ be all the connected components of $G[U]$. By \cref{claim:6}, the vertices in $U_3$ belong to the same connected component. Without loss of generality, we assume $U_3 \subseteq V(G_1)$.  By \eqref{eq3.2}, we have $|E(G)| \geq 10 + 3|U_3| + 4|U_4| + \sum_{i=1}^m (|V(G_i)| - 1) = 4n-10 + |U_4| - m$.

\noindent \textbf{Subcase 2.2.1:} $V(G_1) \cap U_4 = \emptyset$.

In this subcase, $U_3=V(G_1)$, $G[U_3]=G_1$, and $m \geq 2$ with $U_4= V(G_2) \cup ...\cup V(G_m)$.

If there exists some $i(2\leq i \leq m)$ such that $|V(G_i)| = 1$, then we take $V(G_i) = \{u\}$. Let $Q$ be the copy of $K_3 \vee P_k$ in $G+vu$. Since $d_{G+vu}(v) = d_{G+vu}(u) = 5 < k+2$, we have $v, u \in PP(Q)$ and $C(Q)\subseteq N(v)$. Without loss of generality, we assume $C(Q) = \{v_1, v_2, v_3\}$. Since $V(G_1) \cap U_4 = \emptyset$, we have $PP(Q)\subseteq \{u, v, v_4\}\cup U_4$. Thus there exists a path $P_{k-3}$ in some $G_j$ of $G[U_4]$, where $j \in \{2, ... , m\} \setminus \{i\}$, and $|U_4| \geq (m-2)+(k-3)$. Then $|E(G)| \geq 4n + k - 15 $ and \eqref{eq3.1} holds by direct checking and $k\geq 4$.

If there exists some $i(2\leq i \leq m)$ such that $|V(G_i)| \geq 3$, then $|U_4| \geq (m-2)+3$ and $k \geq 6$ since $G[N[v] \cup U_4]$ contains a copy of $K_3 \vee P_5$. Then $|E(G)| \geq 4n-9$ and \eqref{eq3.1} holds by direct checking and $k\geq 6$. 

If $|V(G_i)| = 2$ for any $i(2\leq i \leq m)$, then $|U_4|\geq 2(m-1)$ and $k \geq 5$ since $G[N[v] \cup U_4]$ contains a copy of $K_3 \vee P_4$. Let $Q^\prime$ be the copy of $K_3 \vee P_k$ in $G+vw$, where $w \in V(G_2)$.  Since $V(G_1) \cap U_4 = \emptyset$, we have $|U_3 \cap V(Q^\prime)| = \emptyset$, $V(Q^\prime)\subseteq N[v]\cup U_4$ and $|U_4 \cap V(Q^\prime)| \geq k-2\geq 3$, which implies $m \geq 1 + \left\lceil \frac{k-2}{2} \right\rceil
 \geq 3$. Then $|E(G)| \geq 4n - 10 + 2(m-1) - m$ and \eqref{eq3.1} holds by direct checking and $k\geq 5$.

\noindent \textbf{Subcase 2.2.2:} $V(G_1) \cap U_4 \neq \emptyset$.

Since $V(G_i) \cap U_4 \neq \emptyset$ for all $i$, we have $|U_4| \geq m$, which implies $|E(G)| \geq 4n - 10$ and $n =|N[v]|+|U_3|+|U_4|\geq 1+4+4+m = 9+m$. Thus \eqref{eq3.1} holds by direct checking.

\subsection{ $\delta=5 \leq \Delta \leq n-2$}

 Firstly, we show \cref{claim:7} holds.

\begin{claim}\label{claim:7}
Let $\delta \geq 4$. Then for any $v_j \in V_{NK_3}$, we have \\
\textup{(i)} if $|V_{K_3}|=3$, then $|E[\{v_j\}, V_{K_3}]| \leq 1$. \\
\textup{(ii)} there exists $v_i\in V_{K_3}$ such that $v_iv_j \notin E(G)$ and $E(G[N(v_i)\cap N(v_j)\cap N(v)])=\emptyset$.
\end{claim}

\noindent \textbf{\textit{Proof of \cref{claim:7}.}} 
If $|V_{K_3}|=3$ and $|E[\{v_j\}, V_{K_3}]| \geq 2$, then $v_j \in V_{K_3}$, a contradiction. Then (i) holds. 

Now we show (ii) holds.
If $v_iv_j \in E(G)$ for any $v_i \in V_{K_3}$, then $v_j \in V_{K_3}$, a contradiction. Therefore, there exists $v_i\in V_{K_3}$ such that $v_iv_j \notin E(G)$. Now we show $E(G[N(v_i)\cap N(v_j)\cap N(v)])=\emptyset$. 

If $|N(v_i) \cap N(v_j) \cap N(v)| \leq 1$, then $E(G[N(v_i)\cap N(v_j)\cap N(v)])=\emptyset$.

If $|N(v_i) \cap N(v_j) \cap N(v)| \geq 2$ and there exist  $v_a, v_b \in N(v_i) \cap N(v_j) \cap N(v)$ such that $v_av_b \in E(G)$, then $G[\{v_j, v_a, v_b\}]\cong K_3$ and $v_j \in V_{K_3}$, a contradiction, which implies $E(G[N(v_i)\cap N(v_j)\cap N(v)]) = \emptyset$. 
\hfill $\square$

By (3) of \cref{claim:1}, $G[N(v)]$ contains a copy of $K_3$, and hence $3 \leq |V_{K_3}| \leq 5= |N(v)|$. Now we show \eqref{eq3.1} holds by the following three cases.

\noindent \textbf{Case 1:} $|V_{K_3}| = 3$.  

In this case, $|V_{NK_3}| = 2$, $|E(G_{K_3})| = 3$ and $3 \leq |E(G[N(v)])| \leq 6$ by (i) of \cref{claim:7}.

\noindent \textbf{Subcase 1.1:} $3\leq |E(G[N(v)])| \leq 5$. 

In this subcase, $0 \leq |E[V_{K_3}, V_{NK_3}]| \leq 2$. 
Now we show there exist $v_i, v_j \in N(v)$ such that $N(v_i) \cap N(v_j) \cap N(v) = \emptyset$. 

If $|E[V_{K_3}, V_{NK_3}]|=0$, then for each $v_i \in V_{K_3}$ and $v_j \in V_{NK_3}$, we have $N(v_i) \cap N(v_j) \cap N(v) = \emptyset$. 

If $|E[V_{K_3}, V_{NK_3}]|= 1$, and $v_sv_t \in E[V_{K_3}, V_{NK_3}]$ with $v_s \in V_{K_3}$, $v_t \in V_{NK_3}$, we take $v_i \in V_{K_3}\setminus \{v_s\}$ and $v_j \in V_{NK_3}\setminus \{v_t\}$, and thus $N(v_i) \cap N(v_j) \cap N(v) = \emptyset$. 

If $|E[V_{K_3}, V_{NK_3}]|=2$, and there exists $v_s \in N(v_i) \cap N(v_j) \cap N(v)$ for $v_i,v_j\in V_{NK_3}$, then $d(v_s)= n-1$ by \cref{claim:2}, a contradiction with $\Delta\leq n-2$. Therefore, $N(v_i) \cap N(v_j) \cap N(v) = \emptyset$ for $v_i, v_j \in V_{NK_3}$.

Let $Q_{i, j}$ be a copy of $K_3  \vee P_k$ in $G+v_iv_j$ and $r_{i, j}=|C(Q_{i, j}) \cap U|$. 
By \cref{claim:4} and \eqref{eq3.6} with $\delta^* = \delta$, we have $2|E(G)|\geq 8n-24+ (k-3)r_{i, j}$ with $1\leq r_{i, j} \leq 3$. Since $n= 6+ \sum_{p=3}^{\delta} |U_p|\geq 6+ \sum_{p=4}^{\delta} |U_p|$, by \eqref{eq3.10}, we have 

\begin{equation} 2|E(G)| \geq \begin{cases} 
8n+k-27,  \text{where $n\geq k+7$}, & \text{if $r_{i, j} = 1$}; \\
8n+2k-30, \text{where $n\geq 9$},& \text{if $r_{i, j} = 2$}; \\
8n+3k-33, \text{where $n\geq 9$},& \text{if $r_{i, j} = 3$}. \label{eq3.12}
\end{cases}
\end{equation}
Thus \eqref{eq3.1} holds for $k\geq 3$ by direct checking.

\noindent \textbf{Subcase 1.2:} $|E(G[N(v)])| = 6$.  

In this subcase, $|E[V_{K_3}, V_{NK_3}]|= 2$ and $|E(G_{NK_3})| = 1$ by (i) of \cref{claim:7}, and  
there exist $v_i \in V_{K_3}$ and $v_j \in V_{NK_3}$ such that $v_iv_j \notin E(G)$ and $E(G[N(v_i)\cap N(v_j)\cap N(v)])=\emptyset$ by (ii) of \cref{claim:7}. Let $Q_{i, j}$ be a copy of $K_3  \vee P_k$ in $G+v_iv_j$ and $r_{i, j}=|C(Q_{i, j}) \cap U|$. Then $0\leq r_{i, j} \leq 3$ by \cref{claim:5}. 

If $r_{i, j} = 0$, by \cref{claim:5} and \eqref{eq3.5} with $\delta^* = \delta$, we have $2|E(G)| \geq 8n+k-26$, where $n\geq k+6$. Thus \eqref{eq3.1} holds for $k\geq 3$ by direct checking. 

If $1\leq r_{i, j} \leq 3$, by \cref{claim:5} and \eqref{eq3.6} with $\delta^* = \delta$, we have $2|E(G)|\geq 8n-22+ (k-3)r_{i, j}$. Since $n= 6+ \sum_{p=3}^{\delta} |U_p|\geq 6+ \sum_{p=4}^{\delta} |U_p|$, by \eqref{eq3.11}, we have 

\begin{equation} 2|E(G)| \geq \begin{cases} 
8n+k-25,  \text{where $n\geq 8$},& \text{if $r_{i, j} = 1$}; \\
8n+2k-28, \text{where $n\geq 8$},& \text{if $r_{i, j} = 2$}; \\
8n+3k-31, \text{where $n\geq 9$},& \text{if $r_{i, j} = 3$}. \label{eq3.13}
\end{cases}
\end{equation}
Thus \eqref{eq3.1} holds for $k\geq 3$ by direct checking.

\noindent \textbf{Case 2:} $|V_{K_3}| = 4$.

In this case, $|V_{NK_3}| = 1$. Without loss of generality, we assume $V_{NK_3} = \{v_1\}$. Since $|V_{K_3}| = 4$, $G_{K_3}$ contains at least two copies of $K_3$, and thus $|E(G[N(v)])| \geq|E(G_{K_3})| \ge 5$. 

Firstly, we show $5 \leq |E(G[N(v)])| \leq 7$.
If $|E(G[N(v)])| \ge 8$, then 
\begin{equation*} |E[\{v_1\}, V_{K_3}]| \geq \begin{cases} 
2,  & \text{if $|E(G_{K_3})| = 6$}; \\
3, & \text{if $|E(G_{K_3})| = 5$}. 
\end{cases}
\end{equation*}
Therefore, there exist $v_a, v_b \in V_{K_3}\cap N(v_1)$ such that $v_av_b\in E(G)$ and thus $G[\{v_a, v_b, v_1\}]\cong K_3$, which implies $v_1 \in V_{K_3}$, a contradiction.

\noindent \textbf{Subcase 2.1:} $5 \leq |E(G[N(v)])| \leq 6$. 

In this subcase, $0 \leq|E[V_{K_3}, V_{NK_3}]| \leq 1$ and $5 \leq |E(G_{K_3})|\leq 6$.
Now we show there exists $v_i, v_j \in N(v)$ such that $N(v_i) \cap N(v_j) \cap N(v) = \emptyset$. 

If $|E[V_{K_3}, V_{NK_3}]|=0$, then for each $v_i \in V_{K_3}$ and $v_1 \in V_{NK_3}$, we have $N(v_i) \cap N(v_1) \cap N(v) = \emptyset$. 

If $|E[V_{K_3}, V_{NK_3}]|=1$, then $|E(G_{K_3})|=5$. We assume that $v_sv_t\notin E(G)$, where $v_s, v_t\in V_{K_3}$. Now we show $v_1v_s\in E(G)$ or $v_1v_t\in E(G)$. Otherwise, $v_1v_k\in E(G)$ with $v_k \in V_{K_3}\setminus \{v_s, v_t\}$. Then $d_{G[N(v)]}(v_k)= \delta-1= 4$ and $G[N(v)\setminus \{v_k\}]$ contains no copy of $K_3$, and thus $d(v_k)= n-1$ by \cref{claim:2}, a contradiction with $\Delta \leq n-2$. Therefore, we take $v_i=v_1$, 
\begin{equation*} v_j = \begin{cases} 
v_s,  & \text{if $v_1v_t\in E(G)$}; \\
v_t, & \text{if $v_1v_s\in E(G)$}, 
\end{cases}
\end{equation*}
and we have $N(v_i) \cap N(v_j) \cap N(v) = \emptyset$. 

Let $Q_{i, j}$ be a copy of $K_3  \vee P_k$ in $G+v_iv_j$ and $r_{i, j}=|C(Q_{i, j}) \cap U|$. 
By \cref{claim:4} and \eqref{eq3.6} with $\delta^* = \delta$, we have $2|E(G)|\geq 8n-24+ (k-3)r_{i, j}$ with $1\leq r_{i, j} \leq 3$. Since $n= 6+ \sum_{p=3}^{\delta} |U_p|\geq 6+ \sum_{p=4}^{\delta} |U_p|$, by \eqref{eq3.10}, we have  \eqref{eq3.12} holds. 
Thus \eqref{eq3.1} holds for $k\geq 3$ by direct checking.

\noindent \textbf{Subcase 2.3:} $|E(G[N(v)])| = 7$.

Since $|E(G_{K_3}| \geq 5$ and $|E(G_{K_3})| + |E[V_{K_3}, V_{NK_3}]|= |E(G[N(v)])|$, we have $2|E(G_{K_3})| + |E[V_{K_3}, V_{NK_3}]| \geq 12$. 

There exist $v_i \in V_{K_3}$ and $v_j \in V_{NK_3}$ such that $v_iv_j \notin E(G)$ and $E(G[N(v_i)\cap N(v_j)\cap N(v)])=\emptyset$ by (ii) of \cref{claim:7}. 
Let $Q_{i, j}$ be a copy of $K_3  \vee P_k$ in $G+v_iv_j$ and $r_{i, j}=|C(Q_{i, j}) \cap U|$. Then $0\leq r_{i, j} \leq 3$ by \cref{claim:5}. 

If $r_{i, j} = 0$, by \cref{claim:5} and \eqref{eq3.5} with $\delta^* = \delta$, we have $2|E(G)| \geq 8n+k-24$, where $n\geq k+6$. Thus \eqref{eq3.1} holds for $k\geq 3$ by direct checking. 

If $1\leq r_{i, j} \leq 3$, by \cref{claim:5} and \eqref{eq3.6} with $\delta^* = \delta$, we have $2|E(G)|\geq 8n-22+ (k-3)r_{i, j}$. Since $n= 6+ \sum_{p=3}^{\delta} |U_p|\geq 6+ \sum_{p=4}^{\delta} |U_p|$, by \eqref{eq3.11}, we have  \eqref{eq3.13} holds. 
Thus \eqref{eq3.1} holds for $k\geq 3$ by direct checking.

\noindent \textbf{Case 3:} $|V_{K_3}| = 5$.

Firstly, we show $|E(G[N(v)])|= |E(G_{K_3})| \ge 8$.

Since $|V_{K_3}| = 5$, $G_{K_3}$ contains at least two copies of $K_3$ and $|E(G_{K_3})| \ge 6$. 
If $G_{K_3}$ contains exactly two copies of $K_3$, then there exists $v_s\in N(v)$ such that $G[N(v)\setminus \{v_s\}]$ contains no copy of $K_3$ and $d_{G[N(v)]}= \delta -1$, which implies $d(v_s)= n-1$ by  \cref{claim:2}, a contradiction with $\Delta\leq n-2$.  
Therefore, $G_{K_3}$ contains at least three copies of $K_3$ and thus $|E(G_{K_3})|=|E(G[N(v)])| \ge 7$.

If $|E(G_{K_3})| = 7$, then by direct checking, $G_{K_3}\in \{G_1, G_2\}$, where the graphs $G_1$ and $G_2$ are shown in Figure \ref{fig:graphs}. Thus there exists $v_s \in N(v)$ satisfying the conditions of \cref{claim:2}, which implies $d(v_s)= n-1$,  a contradiction with $\Delta\leq n-2$. 
Therefore, $|E(G[N(v)])|= |E(G_{K_3})| \ge 8$.

Now, we complete the proof of Case 3 by the following two subcases.

\begin{figure}[htbp] 
    \centering
\begin{tikzpicture}[
    scale=0.8, 
    dot/.style={circle, fill=black, inner sep=1.5pt}
]

    \begin{scope}[xshift=0cm]
        \node[dot] (v1) at (-2.5, 0) {};
        \node[dot] (v2) at (-1.2, 0) {};
        \node[dot] (v3) at (0, 0) {};
        \node[dot] (v4) at (1.2, 0) {};
        \node[dot] (v5) at (2.5, 0) {};

        \draw (v1) -- (v2) -- (v3) -- (v4) -- (v5); 
        \draw (v1) to[bend left=40] (v3); 
        \draw (v2) to[bend right=40] (v4); 
        \draw (v3) to[bend left=40] (v5); 
        
        \node[below] at (v3) {$v_s$};
        \node[below] at (0, -1.5) {$G_1$};
    \end{scope}

    \begin{scope}[xshift=6cm]
        \node[dot] (v1) at (-2.5, 0) {};
        \node[dot] (v2) at (-1.2, 0) {};
        \node[dot] (v3) at (0, 0) {};
        \node[dot] (v4) at (1.2, 0) {};
        \node[dot] (v5) at (2.5, 0) {};

        \draw (v1) -- (v2) -- (v3) -- (v4) ; 
        \draw (v1) to[bend left=40] (v3); 
        \draw (v2) to[bend right=40] (v4); 
        \draw (v3) to[bend left=40] (v5); 
        \draw (v2) to[bend left=60] (v5); 

        \node[below] at (v2) {$v_s$};
        \node[below] at (0, -1.5) {$G_2$};
    \end{scope}

\end{tikzpicture}
\caption{The graphs $G_1$ and $G_2$.} 
    \label{fig:graphs} 
\end{figure}
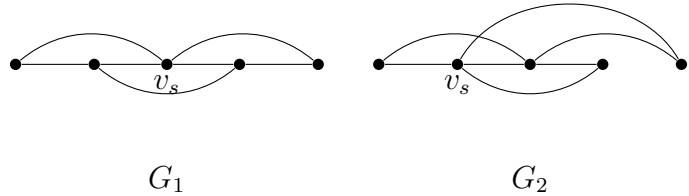

\noindent \textbf{Subcase 3.1:} $|E(G[N(v)])| \geq 9$.

By \eqref{eq3.8}, we have 
$2|E(G)| \ge 8n - 20$ and \eqref{eq3.1} holds for $k\geq 3$ by direct checking.

\noindent \textbf{Subcase 3.2:} $|E(G[N(v)])| = 8$.

In this subcase, $G_{K_3}\in \{G_3, G_4\}$, where $G_3$ and $G_4$ are obtained from $K_5$ by deleting two adjacent edges and two non-adjacent edges, respectively.

\noindent \textbf{Subcase 3.2.1:} $G_{K_3} \cong G_4$.

Let $v_s \in V(G_4)$ with $d_{G_2}(v_s) = 4$. Then $G_4 \setminus \{v_s\}$ contains no copy of $K_3$. By \cref{claim:2}, we have $d(v_s) = n - 1$, a contradiction with $\Delta \leq n-2$. 

\noindent \textbf{Subcase 3.2.2:} $G_{K_3} \cong G_3$. 

Without loss of generality, we assume $d_{G[N(v)]}(v_1) = 2$ and $v_1 v_2, v_1 v_3 \notin E(G)$.
Since $d(v_1)\geq \delta= 5$, we have $|U|\geq 2$, and thus $n\geq 8$.
Let $Q_{1,2}$ be a copy of $K_3  \vee P_k$ in $G+v_1v_2$ and $r_{1,2}=|C(Q_{1,2}) \cap U|$.
By \eqref{eq3.5} with $\delta^*=\delta$, we have 
$2|E(G)| \ge 8n - 22 + (k - 3)r_{1,2} + |U_4| + 2|U_5|$.

If $k=3$, then $|E(G)| \ge 4n - 11$, and thus \eqref{eq3.1} holds.

If $k\geq 4$, then we only need to show $|E(G)| \ge 4n - 10$ if $|U_5| =0$ and $r_{1,2}= 0$. 
Clearly, $v_1,v_2\in V(Q_{1,2})$, $v\notin C(Q_{1,2})$ since $d_{G+v_1v_2}(v)=5 < k+2$, $C(Q_{1,2}) \subseteq \{v_1, v_2, v_4, v_5\}$ by (4) of Proposition \ref{proposition:2.5}, and $|V(Q_{1,2})\cap U|\geq k-3$. Then there exists $w \in PP(Q)$ such that $wv_1, wv_2 \in E(G)$, which implies $w \in U_{\geq4}$ by \cref{claim:3}. Thus $|E(G)| \geq 4n - 10$ and \eqref{eq3.1} holds for $k\geq 4$ by direct checking.

\subsection{$\delta=6 \leq \Delta \leq n-2$}

Firstly, we show \cref{claim:8} holds.

\begin{claim} \label{claim:8}
If there exist $m$ distinct pairs $\{v_i, v_{i^\prime}\}$ of vertices in N(v) such that $v_iv_{i^\prime} \notin E(G)$ and $N(v_i) \cap N(v_{i^\prime}) \cap N(v) = \emptyset$, where $1\leq m \leq 3$, then we have \\
\textup{(i)} $|N( v_i)\cap N(v_{i^\prime}) \cap U_{\geq4}|\geq 3$. \\
\textup{(ii)} if $\{v_a, v_{a^\prime}, v_b, v_{b^\prime}\}\subseteq N(v)$ such that $N(v_i) \cap N(v_{i^\prime}) \cap N(v) = \emptyset$, where $i\in \{a, b\}$, and $u \in N( v_a)\cap N(v_{a^\prime}) \cap N( v_b)\cap N(v_{b^\prime})\cap U_{\geq 4} \neq \emptyset$, then $u \in U_{\geq 5}$.\\
\textup{(iii)} $\sum_{p=4}^\delta (p-3)|U_p| \ge 3m$. 
\end{claim}

\noindent \textbf{\textit{Proof of \cref{claim:8}.}} 
It is clear that (i) holds by \cref{claim:4}.
 
Now we show (ii) holds. Since $N(v_a) \cap N(v_{a^\prime}) \cap N(v) = \emptyset$ and $N(v_b) \cap N(v_{b^\prime}) \cap N(v) = \emptyset$, we have $G[\{v_a, v_b, v_{a^\prime}, v_{b^\prime}\}]$ contains no copy of $K_3$. Therefore, by (1) of \cref{claim:1}, we have $X_{a a^\prime b b^\prime} = \emptyset$, which implies $u \in U_{\geq 5}$.

Now we show (iii) holds by investigating the three cases where $m=1, 2,$ and $3$. 

If $m = 1$, then (iii) holds by \cref{claim:4}.

If $m = 2$, we take $\{v_a, v_{a^\prime}, v_b, v_{b^\prime}\}\subseteq N(v)$ such that $N(v_i) \cap N(v_{i^\prime}) \cap N(v) = \emptyset$, where $i\in \{a, b\}$. 
By (ii), we have $N( v_a)\cap N(v_{a^\prime}) \cap N( v_b)\cap N(v_{b^\prime})\cap U_4 = \emptyset$. We now assume that $|N( v_a)\cap N(v_{a^\prime}) \cap N( v_b)\cap N(v_{b^\prime})\cap U_{\geq5}| = \ell$. 

When $\ell \geq 3$, then  $\sum_{p=4}^\delta (p-3)|U_p| \geq 2|U_{\geq 5}| \geq 6$, so (iii) holds. 
When $0\leq \ell < 3$, then $|U_4|\geq |N( v_a)\cap N(v_{a^\prime}) \cap U_4| + |N( v_b)\cap N(v_{b^\prime}) \cap U_4| \geq (3-\ell) + (3-\ell) =6-2\ell$, and thus $\sum_{p=4}^\delta (p-3)|U_p| \geq |U_4| + 2|U_{\geq5}|\geq 6-2\ell +2\ell =6$, so (iii) holds. 

If $m = 3$, we take $\{v_a, v_{a^\prime}, v_b, v_{b^\prime}, v_c, v_{c^\prime}\} \subseteq N(v)$ such that $N(v_i) \cap N(v_{i^\prime}) \cap N(v) = \emptyset$, where $i\in \{a, b, c\}$. 
Let $|N( v_a)\cap N(v_{a^\prime}) \cap N( v_b)\cap N(v_{b^\prime}) \cap N( v_c)\cap N(v_{c^\prime}) \cap U_{\geq6}|=s$, $|N(v_x) \cap N(v_{x^\prime}) \cap N(v_y) \cap N(v_{y^\prime}) \cap U_5|= s_z$, where $\{x,y,z\} \subseteq \{a,b,c\}$. Then $|U_5|\geq s_a + s_b + s_c$. Now we show $\sum_{p=4}^\delta (p-3)|U_p| \ge 9$.

If $s+s_x+s_y \geq 3$ for any $\{x,y\} \subseteq \{a,b,c\}$, then  $\sum_{p=4}^\delta (p-3)|U_p| \geq 2|U_5|+ 3|U_{\geq6}|\geq 2(s_a+s_b+s_c)+3s \geq 9$. 

If $s+s_a+s_b \geq 3$, $s+s_a+s_c \geq 3$ and $s+s_b+s_c < 3$, then $|U_4|\geq |N( v_a)\cap N(v_{a^\prime}) \cap U_4|\geq 3-s-s_b-s_c$, and thus $\sum_{p=4}^\delta (p-3)|U_p| \geq |U_4| + 2|U_5| +3|U_{\geq6}|\geq (3-s-s_b-s_c) +2(s_a+s_b+s_c)+3s = 3+ (s+s_a+s_b)+ (s+s_a+s_c)\geq 9$. 

If $s+s_a+s_b \geq 3$, $s+s_a+s_c < 3$ and $s+s_b+s_c < 3$, then $|U_4|\geq |N( v_a)\cap N(v_{a^\prime}) \cap U_4|+ |N(v_b)\cap N(v_{b^\prime}) \cap U_4|\geq (3-s-s_b-s_c)+ (3-s-s_a-s_c)$, and thus $\sum_{p=4}^\delta (p-3)|U_p| \geq |U_4| + 2|U_5| +3|U_{\geq6}|\geq (3-s-s_b-s_c)+ (3-s-s_a-s_c) +2(s_a+s_b+s_c)+3s = 3+ 3+ (s+s_a+s_b)\geq 9$. 

If $s+s_x+s_y < 3$ for any $\{x,y\} \subseteq \{a,b,c\}$, then $|U_4|\geq |N( v_a)\cap N(v_{a^\prime}) \cap U_4| + |N( v_b)\cap N(v_{b^\prime}) \cap U_4|+ |N( v_c)\cap N(v_{c^\prime}) \cap U_4| \geq (3-s-s_b-s_c)+ (3-s-s_a-s_c)+ (3-s-s_a-s_b) =9-2(s_a+s_b+s_c)-3s$, and thus $\sum_{p=4}^\delta (p-3)|U_p| \geq |U_4| + 2|U_5| +3|U_{\geq6}|\geq [9-2(s_a+s_b+s_c)-3s] +2(s_a+s_b+s_c)+3s =9$. 

Therefore, we complete the proof.
\hfill $\square$

By (3) of \cref{claim:1}, $G[N(v)]$ contains a copy of $K_3$, so $3 \leq |V_{K_3}| \leq 6= |N(v)|$. Now we show \eqref{eq3.1} holds by the following four cases. 

\noindent \textbf{Case 1:} $|V_{K_3}| = 3$. 

In this case, $|E(G_{K_3})| = |V_{NK_3}|=3$. By (i) of \cref{claim:7}, we have $|E[\{v_j\}, V_{K_3}]| \le 1$ for any $v_j \in V_{NK_3}$, and then $d_{G[N[v]]}(v_j) \leq 3$. Since $d(v_j)\geq \delta = 6$, we have $|U|\geq 3$ and $n \ge  10$.
By \eqref{eq3.9}, we have $2|E(G)| \geq 9n - 30$ and \eqref{eq3.1} holds by direct checking.

\noindent \textbf{Case 2:} $|V_{K_3}| = 4$.

In this case, we have $|E(G_{K_3})| \geq 5$ and $G_{K_3}$ contains at least two copies of $K_3$. 

If $|E(G[N(v)])| \le 8$, then $\sum_{i=1}^6 d_{G[N(v)]}(v_i) \leq 16$, and thus there exists some $v_j \in N(v)$ such that $d_{G[N(v)]}(v_j) \le 2$. Since $d(v_j)\geq \delta =6$, we have $|U|\geq 3$, which implies $n \geq  10$.
Especially, if $k\geq 6$, then $n\geq a_k+3 \geq 13$ by the definition of $a_k$.
By \eqref{eq3.9}, we have $2|E(G)| \geq 9n - 31$ and \eqref{eq3.1} holds by direct checking.

If $|E(G[N(v)])| \ge 9$, then $|E(G_{K_3})| + |E[V_{K_3}, V_{NK_3}]| \geq 8$ and $2|E(G_{K_3})| + |E[V_{K_3}, V_{NK_3}]|\\ \geq 13$ since $|E(G_{NK_3})| \le 1$.
For any $v_j \in V_{NK_3}$, we have $|E[\{v_j\}, V_{K_3}]| \le 2$. Otherwise, there exist $v_a, v_b \in V_{K_3}\cap N(v_j)$ such that $v_av_b\in E(G)$, and thus $G[\{v_a, v_b, v_j\}]\cong K_3$, which implies $v_j \in V_{K_3}$, a contradiction.
Thus $d_{G[N(v)]}(v_j) \le 3$ and $|U|\geq 2$ since $d(v_j)\geq \delta =6$, which implies $n \ge  9$.
By \eqref{eq3.9}, we have $2|E(G)| \geq 9n - 28$ and \eqref{eq3.1} holds by direct checking.

\noindent \textbf{Case 3:} $|V_{K_3}| = 5$.

In this case, we have $6 \leq |E(G_{K_3})|\leq 10$ and $G_{K_3}$ contains at least two copies of $K_3$. Let $V_{NK_3}= \{v_j\}$. Now we study the following three subcases.

\noindent \textbf{Subcase 3.1:} $|E(G[N(v)])| = 6$.

In this subcase, we have $|E(G_{K_3})| = |E(G[N(v)])|=6$, $|E[\{v_j\}, V_{K_3}]| = 0$, $|U|\geq 5$ since $d(v_j)\geq \delta =6$, and $n \geq 12$. 

If $3\leq k \leq 5$, then $2|E(G)| \geq 9n - 34$ by \eqref{eq3.9}, and thus \eqref{eq3.1} holds by direct checking.

If $k\geq 6$, then $k+2-\delta >0$ and $N(v_i) \cap N(v_j) \cap N(v) = \emptyset$, where $v_i \in V_{K_3}$. Let $Q_{i, j}$ be a copy of $K_3  \vee P_k$ in $G+v_iv_j$ and $r_{i, j}=|C(Q_{i, j}) \cap U|$. 
By \cref{claim:4} and \eqref{eq3.6} with $\delta^* = \delta$, we have $2|E(G)|\geq 9n-34+ (k-4)r_{i, j}$ with $1\leq r_{i, j} \leq 3$. Thus $2|E(G)|\geq 9n+k-38$ and \eqref{eq3.1} holds for $k \geq 6$ by direct checking.

\noindent \textbf{Subcase 3.2:} $|E(G[N(v)])| = 7$.

In this subcase, we have $6\leq |E(G_{K_3})|\leq 7$ and $0 \leq |E[V_{K_3}, V_{NK_3}]|\leq 1$, which implies $2|E(G_{K_3})| + |E[V_{K_3}, V_{NK_3}]| \ge 13$. 
Since $d(v_j)\geq \delta =6$, we have $|U|\geq 4$ and $n \geq 11$. Especially, if $k\geq 6$, then $n\geq a_k+3 \geq 13$ by the definition of $a_k$.
By \eqref{eq3.9}, we have $2|E(G)| \geq 9n - 33$ and \eqref{eq3.1} holds by direct checking.

\noindent \textbf{Subcase 3.3:} $|E(G[N(v)])| \geq 8$.

If $|E(G[N(v)])| =8$, then $2|E(G_{K_3})|+ |E[V_{K_3}, V_{NK_3}]|= |E(G_{K_3})| + |E(G[N(v)])| \ge 6 + 8 = 14$ and $|E[V_{K_3}, V_{NK_3}]| \le 2$. Thus $d_{G[N[v]]}(v_j) \leq 3$, $|U|\geq 3$ and $n \geq 10$.

If $|E(G[N(v)])| \geq 9$, then $2|E(G_{K_3})|+ |E[V_{K_3}, V_{NK_3}]|=|E(G_{K_3})| + |E(G[N(v)])| \ge 6 + 9 = 15$ and $|E[V_{K_3}, V_{NK_3}]| \le 3$. Otherwise, if $|E[V_{K_3}, V_{NK_3}]| \geq 4$, there exist $v_a, v_b\in V_{K_3}$ such that $v_av_b, v_av_j, v_jv_b \in E(G)$, which implies $v_j \in V_{K_3}$, a contradiction. Thus $d_{G[N[v]]}(v_j) \leq 4$, $|U|\geq 2$ and $n \geq 9$.

If $3\leq k \leq 4$, then $2|E(G)| \geq 9n - 32$ by \eqref{eq3.9}, and thus \eqref{eq3.1} holds by direct checking.

If $k\geq 5$, then $k+2-\delta >0$. By (ii) of \cref{claim:7}, there exists $v_i \in V_{K_3}$ such that $v_iv_j \notin E(G)$ and $E(G[N(v_i)\cap N(v_j)\cap N(v)])=\emptyset$. 
Let $Q_{i, j}$ be a copy of $K_3  \vee P_k$ in $G+v_iv_j$ and $r_{i, j}=|C(Q_{i, j}) \cap U|$. 
By \cref{claim:5} and \eqref{eq3.6} with $\delta^* = \delta$, we have  
\begin{equation*} 2E(G) \geq \begin{cases} 
9n-32+ (k-4)r_{i, j},  & \text{if $|E(G[N(v)])| = 8$}; \\
9n-31+ (k-4)r_{i, j},  & \text{if $|E(G[N(v)])| \geq 9$} \\
\end{cases}
\end{equation*}
with $0\leq r_{i, j} \leq 3$ and $n\geq k+7$ if $r=0$. 
Especially, if $k\geq 6$, then $n\geq a_k+3 \geq 13$ by the definition of $a_k$.
Thus \eqref{eq3.1} holds by direct checking.

\noindent \textbf{Case 4:} $|V_{K_3}| = 6$.

Since $|V_{K_3}| = 6$, we have $d_{G[N(v)]}(v_s) \ge 2$ for any $v_s \in N(v)$, which implies $|E(G[N(v)])| \\
\ge 6$. 
For $v_s \in N(v)$, we let  
\[
N^l(v_s) = \{u\in N(v)\, :\, d_{G[N(v)]}(u,v_s) = l\}.
\]

\noindent \textbf{Subcase 4.1:} $|E(G[N(v)])| = 6$.

In this subcase, $G[N(v)] \cong 2K_3$. Without loss of generality, we assume $G[\{v_1,v_2,v_3\}]\cong K_3$ and $G[\{v_4,v_5,v_6\}]\cong K_3$. 
Since $v_iv_{i+3}\notin E(G)$, $N(v_i) \cap N(v_{i+3}) \cap N(v) = \emptyset$ for $i \in \{1, 2, 3\}$, by (i) of \cref{claim:8}, we have $\sum_{p=4}^6 |U_p| \ge 3$ and $n\geq 10$.
By \eqref{eq3.8}, we have $2|E(G)| \geq 9n - 30$ and \eqref{eq3.1} holds by direct checking.

\noindent \textbf{Subcase 4.2:} $|E(G[N(v)])| = 7$. 

In this subcase, $G[N(v)]$ contains six vertices and seven edges, and each vertex is a $K_3$-vertex.
Without loss of generality, we assume $G[\{v_1, v_2, v_3\}] \cong K_3$, $G[\{v_4, v_5, v_6\}] \cong K_3$, and $v_3 v_6 \in E(G)$. 
Since $v_iv_{i+3}\notin E(G)$, $N(v_i) \cap N(v_{i+3}) \cap N(v) = \emptyset$ for $i \in \{1, 2\}$, by (iii) of \cref{claim:8}, we have $\sum_{p=4}^6 (p-3)|U_p| \ge 6$. Since $d_{G[N[v]]}(v_1) = 3$ and $d(v_1)\geq \delta =6$, we have $|U|\geq 3$, and thus $n \ge 10$. Especially, if $k\geq 6$, then $n\geq a_k+3 \geq 13$ by the definition of $a_k$.
By \eqref{eq3.8}, we have $2|E(G)| \geq 9n - 31$ and \eqref{eq3.1} holds by direct checking.

\noindent \textbf{Subcase 4.3:} $|E(G[N(v)])| = 8$.

Clearly, $G[N(v)]$ is not complete. Now we show there exist two vertices $v_i, v_j \in N(v)$ such that $v_i v_j \notin E(G)$ and $N(v_i) \cap N(v_j) \cap N(v) = \emptyset$. Otherwise, we have $d_{G[N(v)]}(v_i, v_j) = 2$ for any two vertices $v_i, v_j \in N(v)$ with $v_i v_j \notin E(G)$ since $N(v_i) \cap N(v_j) \cap N(v) \neq \emptyset$, and there exists a vertex with degree two in $G[N(v)]$ since $|E(G[N(v)])| = 8$, and $|V_{K_3}|=|N(v)|=6$. 
Without loss of generality, we take $d_{G[N(v)]}(v_1) = 2$ and $N^1(v_1) = \{v_2, v_3\}$, and thus $N^2(v_1) = \{v_4, v_5, v_6\}$.
Since $v_1 \in V_{K_3}$, we have $v_2 v_3 \in E(G)$. Since $|E[\{v_1\}, N^1(v_1)]| =2$, $|E[N^1(v_1), N^2(v_1)]| \geq 3$, we have $|E(G[N^2(v_1)])| \le 2$.

If $|E(G[N^2(v_1)])| = 2$, then $|E[N^1(v_1), N^2(v_1)]| = 3$. Without loss of generality, we assume $v_2 v_4 \in E(G)$, then $v_2 v_5, v_2 v_6 \in E(G)$ since $v_5, v_6 \in V_{K_3}$. By \cref{claim:2}, we have $d(v_2) = n - 1$, a contradiction with $\Delta \leq n-2$.

If $|E(G[N^2(v_1)])| = 1$, then $|E[N^1(v_1), N^2(v_1)]| = 4$. Without loss of generality, we assume $v_4 v_5 \in E(G)$ and $v_2 v_4 \in E(G)$, and thus $v_2 v_6 \in E(G)$ since $N(v_4) \cap N(v_6) \cap N(v) \neq \emptyset$, and thus $v_2 v_5 \in E(G)$ since $N(v_5) \cap N(v_6) \cap N(v) \neq \emptyset$. By \cref{claim:2}, we have $d(v_2) = n - 1$, a contradiction with $\Delta \leq n-2$. 

If $|E(G[N^2(v_1)])| = 0$, then $v_2 v_4, v_2 v_5, v_2 v_6, v_3 v_4, v_3 v_5, v_3 v_6 \in E(G)$ since $v_4, v_5, v_6 \in V_{K_3}$, which implies $|E(G[N(v)])| \geq 9$, a contradiction.

By the above arguments, there exist two vertices $v_i, v_j \in N(v)$ such that $v_i v_j \notin E(G)$ and $N(v_i) \cap N(v_j) \cap N(v) = \emptyset$, which implies $\sum_{p=4}^{\delta} (p-3)|U_p|\geq \sum_{p=4}^{\delta} |U_p| \ge 3$ and $n\geq 10$ by \cref{claim:4}. 

If $3\leq k\leq 4$, then $2|E(G)| \geq 9n - 32$ by \eqref{eq3.8}, and \eqref{eq3.1} holds by direct checking. 

If $k\geq 5$, then $k + 2 - \delta > 0$. Let $Q_{i, j}$ be a copy of $K_3  \vee P_k$ in $G+v_iv_j$ and $r_{i, j}=|C(Q_{i, j}) \cap U|$. 
By \cref{claim:4} and \eqref{eq3.5} with $\delta^* = \delta$, we have $2|E(G)|\geq 9n-32+ (k-4)r_{i, j}$ with $1\leq r_{i, j} \leq 3$. 
Then $2|E(G)| \geq 9n + k - 36$ and \eqref{eq3.1} holds by direct checking.

\noindent \textbf{Subcase 4.4:} $|E(G[N(v)])| = 9$.

Now we complete the proof of Subcase 4.4 by the following two Subcases.

\noindent \textbf{Subcase 4.4.1:} There exist some two vertices $v_i, v_j \in N(v)$ such that $v_i v_j \notin E(G)$ and $N(v_i) \cap N(v_j) \cap N(v) = \emptyset$.

By \cref{claim:4}, we have $\sum_{p=4}^{\delta} (p-3)|U_p|\geq \sum_{p=4}^{\delta} |U_p| \ge 3$, so $n\geq 10$. 

If $3\leq k\leq 4$, then $2|E(G)| \geq 9n - 30$ by \eqref{eq3.8}, and \eqref{eq3.1} holds by direct checking. 

If $k\geq 5$, then $k + 2 - \delta > 0$. Let $Q_{i, j}$ be a copy of $K_3  \vee P_k$ in $G+v_iv_j$ and $r_{i, j}=|C(Q_{i, j}) \cap U|$. 
By \cref{claim:4} and \eqref{eq3.5} with $\delta^* = \delta$, we have $2|E(G)|\geq 9n-30+ (k-4)r_{i, j}$ with $1\leq r_{i, j} \leq 3$. 
Then $2|E(G)| \geq 9n + k - 34$ and \eqref{eq3.1} holds by direct checking.

\noindent \textbf{Subcase 4.4.2:} For any two vertices $v_i, v_j \in N(v)$ with $v_i v_j \notin E(G)$, $N(v_i) \cap N(v_j) \cap N(v) \neq \emptyset$.

Clearly, $d_{G[N(v)]}(v_i, v_j) = 2$ for any $v_i, v_j \in N(v)$ with $v_i v_j \notin E(G)$.
Let $d_{G[N(v)]}(v_1) = \min \{d_{G[N(v)]}(v_s)| v_s\in N(v)\}$. Clearly, $2 \leq d_{G[N(v)]}(v_1) \leq 3$ since $|E(G[N(v)])| = 9$.

\noindent \textbf{Subcase 4.4.2.1:} $d_{G[N(v)]}(v_1) = 3$.

It is evident that $d_{G[N(v)]}(v_s) = 3$ for any $v_s \in N(v)$, and thus $|U|\geq 2$, $n\geq 9$. Without loss of generality, we assume $N^1(v_1) = \{v_2, v_3, v_4\}$ and $N^2(v_1) = \{v_5, v_6\}$. 
If $v_5 v_6 \notin E(G)$, then $|E[N^1(v_1), N^2(v_1)]| = 6$ since $d_{G[N(v)]}(v_5) = d_{G[N(v)]}(v_6) = 3$, which implies $|E(G[\{v_2, v_3, v_4\}])|=0$ and $|V_{K_3}|=0$, a contradiction. 
Therefore, $v_5 v_6 \in E(G)$, and thus $|E[N^1(v_1), N^2(v_1)]| = 4$ since $d_{G[N(v)]}(v_5) = d_{G[N(v)]}(v_6) = 3$, and $|E(G[N^1(v_1)])| = 1$ since $|E(G_{K_3})| = 9$. Without loss of generality, we assume $v_2 v_3 \in E(G)$. Then $v_4 v_5, v_4 v_6 \in E(G)$ since $v_4 \in V_{K_3}$. Thus $E(G[N(v_2) \cap N(v_4) \cap N(v)]) = \emptyset$. By \cref{claim:5}, we have $\sum_{p=4}^6 (p-3)|U_p| \ge 2$. 

If $3\leq k\leq 4$, then $2|E(G)| \geq 9n - 31$ by \eqref{eq3.8}, and \eqref{eq3.1} holds by direct checking. 

If $k\geq 5$, then $k + 2 - \delta > 0$. Let $Q_{2, 4}$ be a copy of $K_3  \vee P_k$ in $G+v_2v_4$ and $r_{2, 4}=|C(Q_{2, 4}) \cap U|$. By \cref{claim:5} and \eqref{eq3.5} with $\delta^* = \delta$, we have $2|E(G)|\geq 9n-33+ \sum_{p=4}^{\delta} (p-3)|U_p|+ r_{2, 4}(k-4)$ with $0\leq r_{2, 4} \leq 3$. Thus $2|E(G)| \geq 9n + k - 35$ by \cref{claim:5}, and \eqref{eq3.1} holds by direct checking.

\noindent \textbf{Subcase 4.4.2.2:} $d_{G[N(v)]}(v_1) = 2$.

Without loss of generality, we assume $N^1(v_1) = \{v_2, v_3\}$ and $N^2(v_1) = \{v_4, v_5, v_6\}$. Clearly, $3\leq |E[N^1(v_1), N^2(v_1)]| \leq 6$. 

If $3 \le |E[N^1(v_1), N^2(v_1)]| \le 5$, then there exists some $v_s \in N^2(v_1)$ such that 
$|E[N^1(v_1), \\ \{v_s\}]|=1$, which implies $E(G[N(v_s) \cap N(v_1) \cap N(v)]) = \emptyset$. This reduces to the same situation as in Subcase 4.4.2.1, and thus \eqref{eq3.1} holds.

If $|E[N^1(v_1), N^2(v_1)]| = 6$, then $|E(G[N^2(v_1)])| = 0$, and $v_2 v_4, v_2 v_5, v_2 v_6, v_3 v_4, v_3 v_5, v_3 v_6 \\ \in E(G)$, which implies $d(v_2) = n - 1$ by \cref{claim:2}, a contradiction with $\Delta\leq n-2$.

\noindent \textbf{Subcase 4.5:} $|E(G[N(v)])| \ge 10$.

If $10 \le |E(G[N(v)])| \le 11$, then $\sum_{s=1}^6 d_{G[N(v)]}(v_s) \leq 22$, and there exists $v_s \in N(v)$ such that $d_{G[N(v)]}(v_s)=2$ or $3$, which implies $n \ge 9$. Especially, if $k\geq 6$, then $n\geq a_k+3 \geq 13$ by the definition of $a_k$.

If $3\leq k\leq 4$, then $2|E(G)| \geq 9n - 31$ by \eqref{eq3.8}, and \eqref{eq3.1} holds by direct checking. 

If $k\geq 5$, then $|V(Q_{i, j}) \cap U| \ge k + 3 - 7 \ge 1$, where $Q_{i, j}$ is a copy of $K_3  \vee P_k$ in $G+v_iv_j$ with $v_i, v_j \in N(v)$ and $v_i v_j \notin E(G)$. Then we can show $|U_{\geq 4}|\geq 1$ by investigating whether $C(Q_{i, j}) \cap U$ is empty.
By \eqref{eq3.8}, we have $2|E(G)|\geq 9n-30$ and \eqref{eq3.1} holds by direct checking.

If $|E(G[N(v)])| \ge 12$, then $2|E(G)| \geq 9n - 27$ by \eqref{eq3.8}, and \eqref{eq3.1} holds by direct checking.

\subsection{$\delta=7 \leq \Delta \leq n-2$}
By (3) of \cref{claim:1}, $G[N(v)]$ contains a copy of $K_3$, then $3 \leq |V_{K_3}| \leq 7= |N(v)|$. 

Since $\delta = 7$, we have $|E(G)| \ge \frac{7}{2}n$, which implies \eqref{eq3.1} holds for $3\leq k \leq 4$. Next, we consider $k \ge 5$. Especially, if $k\geq 6$, we have $n\geq a_k+3 \geq 13$ by the definition of $a_k$.

If $\sum_{s=1}^7 d_{G[N(v)]}(v_s) \leq 7t-1$, then there exists $v_i\in N(v)$ such that $d_{G[N(v)]}(v_i) \leq t-1$, which implies $|U|\geq \delta- t$ and $n=1+\delta+|U|\geq 1+2\delta-t$ for $3\leq t\leq 5$.

If $|E(G[N(v)])|\ge 18$, then $2|E(G)| \geq 10n - 30$ by \eqref{eq3.8}, and \eqref{eq3.1} holds by direct checking.
If $14 \le |E(G[N(v)])| \le 17$, then $t=5$, $n\geq 10$, and $2|E(G)| \ge 10n - 38$ by \eqref{eq3.8}, and thus \eqref{eq3.1} holds by direct checking.
If $11 \le |E(G[N(v)])| \le 13$, then $t=4$, $n\geq 11$, and $2|E(G)| \ge 10n - 44$ by \eqref{eq3.8}, and thus \eqref{eq3.1} holds by direct checking.
If $|E(G[N(v)])| = 10$, then $t=3$, $n\geq 12$, and $2|E(G)| \ge 10n - 46$ by \eqref{eq3.8}, and thus \eqref{eq3.1} holds by direct checking.

In the following, we consider $|E(G[N(v)])| \le 9$. 
Clearly, there exists a vertex $v_i$ such that $d_G[N(v)](v_i)\leq 2$, and $n\geq 12$ for $k\geq 5$, $n\geq 13$ for $k\geq 6$. Now we complete the proof of $\delta=7$ by the following three cases.

\noindent \textbf{Case 1:} $3 \leq |V_{K_3}| \leq 5$.

In this case, $|E(G_{K_3})| = 3$ for $|V_{K_3}|=3$; $|E(G_{K_3})| \ge 5$ for $|V_{K_3}|=4$; $|E(G_{K_3})| \ge 6$ for $|V_{K_3}|=5$. By \eqref{eq3.9}, we have 
\begin{equation*} 2E(G) \geq \begin{cases} 
9n-36,  & \text{if $|V_{K_3}| = 3$}; \\
9n-38,  & \text{if $|V_{K_3}| = 4$}; \\
9n-42, & \text{if $|V_{K_3}| = 5$}. \\
\end{cases}
\end{equation*}
Thus \eqref{eq3.1} holds by direct checking.

\noindent \textbf{Case 2:} $|V_{K_3}| = 6$.

In this case, $|E(G_{K_3})| \ge 6$ and $G_{K_3}$ contains at least two copies of $K_3$.

\noindent \textbf{Subcase 2.1:} $|E(G[N(v)])| = 6$.

In this subcase, $G_{K_3} \cong 2K_3$. Without loss of generality, we assume $G[\{v_1,v_2,v_3\}]\cong K_3$ and $G[\{v_4,v_5,v_6\}]\cong K_3$. Since $v_iv_{i+3}\notin E(G)$, $N(v_i) \cap N(v_{i+3}) \cap N(v) = \emptyset$ for $i \in \{1, 2, 3\}$, by (iii) of \cref{claim:8}, we have $\sum_{p=4}^7 (p-3)|U_p| \ge 9$ . 
By \eqref{eq3.8}, we have $2|E(G)| \geq 10n - 45$ and \eqref{eq3.1} holds by direct checking.

\noindent \textbf{Subcase 2.2:} $|E(G[N(v)])| = 7$.

In this subcase, $0\leq |E[V_{K_3}, V_{NK_3}]| \leq 1$ since $|E(G_{K_3})| \ge 6$. 

Now we show there exist two distinct pairs $\{v_i, v_{i^\prime}\}$ of vertices in $N(v)$ such that $v_iv_{i^\prime}\notin E(G)$ and $N(v_i) \cap N(v_{i^\prime}) \cap N(v) = \emptyset$. 
If $|E[V_{K_3}, V_{NK_3}]| = 1$, then $G_{K_3} \cong 2K_3$, and thus the result follows. 
If $|E[V_{K_3}, V_{NK_3}]| = 0$, then $G_{K_3}$ contains six vertices and seven edges, and we can assume
 $G[\{v_1, v_2, v_3\}] \cong K_3$, $G[\{v_4, v_5, v_6\}] \cong K_3$ and $v_3 v_6 \in E(G)$. Therefore,  $v_iv_{i+3}\notin E(G)$, $N(v_i) \cap N(v_{i+3}) \cap N(v) = \emptyset$ for $i \in \{1, 2\}$. 
By (iii) of \cref{claim:8}, $\sum_{p=4}^7 (p-3)|U_p| \ge 6$. Together with \eqref{eq3.8}, we have $2|E(G)| \geq 10n - 46$ and \eqref{eq3.1} holds by direct checking.

\noindent \textbf{Subcase 2.3:} $8 \le |E(G[N(v)])| \le 9$.

In this subcase, since $|E(G_{K_3})| \ge 6$, we have $2|E(G_{K_3})| + |E[V_{K_3}, V_{NK_3}]|= |E(G_{K_3})| + |E(G[N(v)])|\geq 14$. 
By \eqref{eq3.9}, we have $2|E(G)| \geq 10n - 46$ and \eqref{eq3.1} holds by direct checking.

\noindent \textbf{Case 3:} $|V_{K_3}| = 7$.

In this case, $8 \leq |E(G_{K_3})| \leq 9$, and thus $G_{K_3}$ contains at least three copies of $K_3$ and at most two components.

\noindent \textbf{Subcase 3.1:} $|E(G[N(v)])| = 8$.

In this case, $G_{K_3}$ contains two components, and one component is $K_3$. Without loss of generality, we assume $G[\{v_1, v_2, v_3\}] \cong K_3$. Then $v_iv_{i+3}\notin E(G)$, $N(v_i) \cap N(v_{i+3}) \cap N(v) = \emptyset$ for $i \in \{1, 2, 3\}$. By (iii) of \cref{claim:8}, we have $\sum_{p=4}^7 (p-3)|U_p| \ge 9$. 
By \eqref{eq3.8}, we have $2|E(G)| \geq 10n - 41$ and \eqref{eq3.1} holds by direct checking.

\noindent \textbf{Subcase 3.2:} $|E(G[N(v)])| = 9$.

If there exist $v_i, v_j \in N(v)$ such that $v_iv_j\notin E(G)$, $N(v_i) \cap N(v_j) \cap N(v) = \emptyset$, then $\sum_{p=4}^7 (p-3)|U_p| \ge 3$ by (iii) of \cref{claim:8}. 
By \eqref{eq3.8}, we have $2|E(G)| \geq 10n - 45$ and \eqref{eq3.1} holds by direct checking.

If $N(v_i) \cap N(v_j) \cap N(v) \neq \emptyset$ for every $v_iv_j\notin E(G)$ and $v_i, v_j \in N(v)$, then $d_{G[N(v)]}(v_i, v_j)\\= 2$. 
Since $|E(G[N(v)])| = 9$ and $|V_{K_3}| =|N(v)|= 7$, there exists $v_s\in N(v)$ such that $d_{G[N(v)]}(v_s) = 2$. 
Without loss of generality, we take $d_{G[N(v)]}(v_1) = 2$. Let $N^1(v_1)  = \{v_2, v_3\}$ and
$N^2(v_1) =\{v_4, v_5, v_6, v_7\}$.
Since $v_1 \in V_{K_3}$, we have $v_2v_3 \in E(G)$. Since $|E[N^1(v_1), N^2(v_1)]| \ge |N^2(v_1)|=4$, we have $|E(G[N^2(v_1)])| \le 2$. 
For any $v_s \in N^2(v_1)$ satisfying $d_{G[N^2(v_1)]}(v_s) = 0$, we have $v_2v_s, v_3v_s \in E(G)$ since $v_s \in V_{K_3}$. Based on the above arguments, we show there exists a contradiction when $0\leq |E(G[N^2(v_1)])|\leq 2$. 

When $|E(G[N^2(v_1)])| = p$ for $p\in \{0,1\}$, we have $|E[N^1(v_1), N^2(v_1)]| \ge 8- 2p$, and thus $|E(G_{K_3})| \ge 11-p$, a contradiction with $|E(G[N(v)])| = 9$. 

When $|E(G[N^2(v_1)])| = 2$ and $G[N^2(v_1)] \cong P_3 \cup K_1$, we have $|E[N^1(v_1), N^2(v_1)]| \ge 5$, and thus $|E(G_{K_3})| \ge 10$, a contradiction with $|E(G[N(v)])| = 9$. 

When $|E(G[N^2(v_1)])| = 2$ and $G[N^2(v_1)] \cong 2P_2$, we have $|E[N^1(v_1), N^2(v_1)]| = 4$ by $|E(G[N(v)])| = 9$. Without loss of generality, we assume $v_4v_5, v_6v_7 \in E(G)$, $v_2v_4, v_2v_5 \in E(G)$ since $v_4, v_5 \in V_{K_3}$.
Then $v_2v_6, v_2v_7 \in E(G)$ since $N(v_4) \cap N(v_6) \cap N(v) \neq \emptyset$ and $N(v_4) \cap N(v_7) \cap N(v) \neq \emptyset$. By \cref{claim:2}, we have $d(v_2) = n-1$, a contradiction with $\Delta\leq n-2$.

Combining the above arguments, we complete the proof of Theorem~\ref{thm1.2}. \hfill $\square$

\section{Conclusion}

In this paper, we investigate the saturation number of $K_3\vee P_k$ for $k\geq 3$ and $n\geq a_k+3$ and show Theorem~\ref{thm1.2} holds by showing \eqref{eq3.1} holds.
In the proof, we adapt the vertex-partition method of \cite{zhang2025},
introduce the sets $V_{K_3}$, $V_{NK_3}$, $U_{\geq p}$ to obtain a further refinement, 
define the notation $G_{K_3}$, $G_{NK_3}$, $r=|C(Q)\cap U|$ and so on, and obtain the lower bound on $|E(G)|$ given by a simpler expression.

Furthermore, natural problems are to characterize all minimal $(K_3\vee P_k)$- saturated graphs and to investigate whether analogous local reductions can determine $\operatorname{sat}(n,K_s\vee P_k)$ for larger fixed values of $s$\cite{zhang2025}. 

\section*{Funding}
This work is supported by the National Natural Science Foundation of China (Grant Nos. 12371347, 12271337).

\section*{Declarations}
\noindent\textbf{Conflict of interest} The authors declare that they have no known competing financial interests or personal relationships that could have appeared to influence the work reported in this paper.\\
\textbf{Data availability} No data was used for the research described in the article.

\end{document}